\documentclass[12pt]{amsart}
\usepackage{latexsym, amsthm, amscd, graphicx, euscript, float, subfig}
\usepackage[english]{babel}
\usepackage{amssymb,amsthm,amsmath}

\newcommand{\PP}{ { \R^2 \setminus\{ 0 \} } }
\newcommand{\PS}{ { \Rn \setminus\{ 0 \} } }
\newcommand{\N}{\mathbb{N}}
\newcommand{\R}{\mathbb{R}}
\newcommand{\C}{\mathbb{C}}
\newcommand{\B}{B}
\newcommand{\D}{\mathbb{D}}
\renewcommand{\H}{\mathbb{H}}

\newcommand{\X}{\mathrm{X}}
\newcommand{\Y}{\mathrm{Y}}

\newcommand{\Hn}{\mathbb{H}^n}
\newcommand{\Bn}{\mathbb{B}^n}

\newcommand{\Rn}{ {\mathbb{R}^n} }
\newcommand{\Ball}[3]{B_{#1}(#2,#3)}

\newcommand{\comment}[1]{}

\newtheorem{theorem}[equation]{Theorem}

\newtheorem{proposition}[equation]{Proposition}

\theoremstyle{remark}
\newtheorem{definition}[equation]{Definition}

\newtheorem{example}[equation]{Example}

\newtheorem{conjecture}{Conjecture}[section]

\numberwithin{equation}{section}

\newcounter{minutes}
\divide\time by 60
\newcounter{hours}
\multiply\time by 60 \addtocounter{minutes}{-\time}

\keywords{hyperbolic metric, quasihyperbolic metric, distance ratio metric, convexity properties, Banach spaces}
\subjclass{Primary  30C65; Secondary  46T05}

\begin{document}

\title[Quasihyperbolic geometry]
{\vspace*{.1cm} Quasihyperbolic Geometry in Euclidean and Banach Spaces}


\author[Riku Kl\'en]{\noindent Riku Kl\'en}
\email{riku.klen@utu.fi}
\address{\newline Department of Mathematics,
\newline FI-20014 University of Turku,
\newline Finland
}

\author[Antti Rasila]{\noindent Antti Rasila}
\email{antti.rasila@iki.fi}
\address{\newline Institute of Mathematics,
\newline Aalto University,
\newline P.O. Box 11100, FI-00076 Aalto,
\newline Finland
}

\author[Jarno Talponen]{\noindent Jarno Talponen}
\email{jarno.talponen@tkk.fi}
\address{\newline Institute of Mathematics,
\newline Aalto University,
\newline P.O. Box 11100, FI-00076 Aalto,
\newline Finland
}

\begin{abstract} 
We consider the quasihyperbolic metric, and its generalizations in both the $n$-dimensional Euclidean space $\Rn$, and in Banach spaces. Historical background, applications, and our recent work on convexity properties of these metrics are discussed.
\end{abstract}

\maketitle

\section{Introduction}

In this paper, we consider the quasihyperbolic metric, and its generalizations in both the $n$-dimensional Euclidean space $\Rn$, and in Banach spaces.

First we give some motivation and historical background. Quasihyperbolic metric can be viewed as a generalization of the hyperbolic metric. We start with an introduction to the hyperbolic metric and related results in function theory.

\subsection{Hyperbolic geometry in plane and in $n$-dimensional Euclidean spaces}

Suppose that $w=f(z)$ is a conformal mapping of the unit disk onto itself. Then by Pick's lemma we have the equality
\[
\bigg|\frac{dw}{dz}\bigg| = \frac{1-|w|^2}{1-|z|^2}.
\]
This identity can be written
\[
\frac{|dw|}{1-|w|^2} = \frac{|dz|}{1-|z|^2},
\]
which means that for any regular curve $\gamma$ in the unit disk $\D=\{z : |z|<1\}$, and for any conformal self mapping $f$ of $\D$, we have
\[
\int_{f\circ\gamma} \frac{|dw|}{1-|w|^2} = \int_{\gamma}  \frac{|dz|}{1-|z|^2}.
\]
We have obtained a length function which is invariant under conformal mappings of the unit disk onto itself. This allows us to give the following definition.

\begin{definition}
\label{diskhyp}
Let $z_1,z_2\in \D$. Then the \emph{hyperbolic distance} $\rho$ of $z_1,z_2$ is defined by
\[
\rho_{\D}(z_1,z_2)=\inf_\gamma \int_{\gamma}  \frac{2|dz|}{1-|z|^2},
\]
where the infimum is taken over all regular curves $\gamma$ connecting $z_1$ and $z_2$.
\end{definition}

Obviously, the multiplier $2$ is harmless, and it is sometimes omitted. By the conformal invariance we can define the hyperbolic distance on any simply connected domain $\Omega\subset\C$ by the formula
\[
\rho_\Omega(z_1,z_2) = \rho_{\D}\big(f(z_1),f(z_2)\big), \qquad z_1,z_2\in \Omega,
\]
where $f\colon \Omega\to \D$ is a conformal mapping. Existence of such mapping is guaranteed by Riemann's mapping theorem.

Note that Definition \ref{diskhyp} makes sense also in $\Rn$ for all $n\ge 2$, if we consider instead of the unit disk $\D$ the unit ball $\Bn=\{ x : |x|<1\}$. However, for $n\ge 3$ we have few conformal mappings, as by the generalized Liouville theorem (see e.g. \cite{vu2}), every conformal mapping of a domain $\Omega$ in the Euclidean space $\Rn$, $n\ge 3$, is a restriction of a M\"obius transformation. Thus, the domains of interest in higher dimensions are essentially the unit ball and the upper half-space $\Hn = \{ x : x_n>0\}$. Because formulas for  M\"obius transformations of the unit ball $\Bn$ onto $\Hn$ are well known, we arrive to the following definitions.

\begin{definition}
The hyperbolic distance between $x,y\in \Bn$ is defined by
$$
\rho_{\Bn}(x,y)=\inf_{\alpha\in\Gamma_{xy}}\int_{\alpha}\frac{2|dx|}{1-|x|^2},
$$
and between $x,y\in \Hn$ by
$$
\rho_{\Hn}(x,y)=\inf_{\alpha\in\Gamma_{ab}}\int_{\alpha}\frac{|dx|}{x_n},
$$
where $\Gamma_{xy}$ is the family of all rectifiable paths joining
$x$ and $y$ in $\Bn$ or $\Hn$, respectively.
\end{definition}

It is well-known that the hyperbolic distance $\rho_\Omega$ is a metric in the topological sense, i.e., the following properties hold for $d=\rho_\Omega$:
\begin{enumerate}
\item $d(x,y) \ge 0$,   
\item $d(x,y) = 0$, if and only if $x = y$,
\item $d(x,y) = d(y,x)$, and
\item $d(x,z) \le d(x,y) + d(y,z)$,
\end{enumerate}
for all $x,y,z\in \Omega$. The last condition is the triangle inequality. A curve $\gamma$ is called a \emph{geodesic} of the metric $d$ if triangle inequality holds as an equality for all points $x,y,z$ on $\gamma$ in that order. 

One may show that hyperbolic metric is always geodesic: for any $x,y\in \Omega$ there is a geodesic $\gamma$ connecting $x,y$, i.e., the infimum in the integral is attained by $\gamma$. For the ball or half-space those geodesics are always circular arcs orthogonal to the boundary. Further calculations allow us to obtain formulas for $\rho_{\Bn}$ and $\rho_{\Hn}$ in the closed form. The hyperbolic metrics in the upper half-plane $\H$ and in the unit disk $\D$ are also given by equations
\begin{equation}
\label{qr218}
\sinh^2\Big(\frac{1}{2}\rho_{\Bn}(x,y)\Big)
= \frac{|x-y|^2}{(1-|x|^2)(1-|y|^2)},
\,\,\,x,y\in\Bn,
\end{equation}
and
\begin{equation}
\cosh\rho_{\Hn}(x,y) = 1+\frac{|x-y|^2}{2\,x_n\,y_n},
\,\,\,x,y\in\Hn.
\end{equation}

Hyperbolic metric is topologically equivalent to the one defined by the restriction of the Euclidean norm: they define the same open sets. In fact, if $\Omega$ is the unit ball, or its image in a M\"obius transformation, the hyperbolic balls
\[
\{ y : \rho_{\Omega}(x,y)<r\},\qquad x\in \Omega,\; r>0,
\]
are Euclidean balls, although their hyperbolic center is not usually the same as the Euclidean center. 

There are certain results which make the hyperbolic metric particularly interesting from the point of view of the geometric function theory. For example, many classical results from the Euclidean geometry have hyperbolic analogues, such as the following hyperbolic form of Pythagoras' Theorem:

\begin{proposition}
\label{pythagoras}
\cite[Theorem 7.11.1]{Beardon}
For a hyperbolic triangle with angles $\alpha,\beta,\pi/2$ and
corresponding hyperbolic opposite side lengths $a,b,c$, we have
\[
\cosh c = \cosh a \cosh b.
\]
\end{proposition}

Here the hyperbolic triangle means the domain whose boundary consists of the hyperbolic geodesics connecting three points in the hyperbolic space. In particular, the hyperbolic distance is very useful in complex analysis. For example, we can prove the following version of the classical Schwarz's lemma (see for example \cite[p. 268]{Gamelin}). Interestingly, this result does also have a complete analogy for quasiconformal mappings in $\Rn$, $n\ge 2$ (see \cite[11.2]{vu2}).

\begin{proposition}
Let $f\colon \D\to\D$ be analytic. Then the following inequality holds
for the hyperbolic distance 
\begin{equation}
\label{pick}
\rho_{\D}\big(f(x),f(y)\big) \leq \rho_{\D}(x,y)\text{ for }x,y\in\D,
\end{equation}
where the equality holds if and only if $f$ is a M\"obius transformation.
\end{proposition}

While the hyperbolic metric is a powerful tool in function theory, this approach also has two important limitations:
\begin{enumerate}
\item In $\Rn$, for $n\ge 3$, the hyperbolic metric cannot be defined for a domain which is not an image of the unit ball in a M\"obius transformation.
\item Even in the plane, there are no actual formulas for hyperbolic distance, except in the case of a few domains for which the Riemann mapping can be found analytically.
\end{enumerate}
We would like to have a metric which could be defined for any domain of interest and which could be easily computed, or at least estimated. Ideally, we would like this metric to be as similar to the hyperbolic metric as possible in other respects. This leads us the quasihyperbolic metric. 

\subsection{Quasihyperbolic and distance ratio metrics}

The quasihyperbolic metric was first introduced by F.W. Gehring and B.P. Palka in 1976 \cite{gp}, and it has been studied by numerous authors thereafter.

\begin{definition}
Let $\Omega$ be a proper subdomain of the Euclidean space $\Rn$, $n \ge 2$. We define the \emph{quasihyperbolic length} of a rectifiable arc $\gamma \subset \Omega$ by
\[
  \ell_k(\gamma) = \int_{\alpha}\frac{|dz|}{d(z,\partial \Omega)}.
\]
The \emph{quasihyperbolic metric} is defined by
\[
  k_\Omega(x,y) = \inf_\gamma \ell_k(\gamma),\index{hyperbolic metric, $k_\Omega$}
\]
where the infimum is taken over all rectifiable curves in $\Omega$ joining $x$ and $y$. If the domain $\Omega$ is clear from the context we use notation $k$ instead of $k_\Omega$.
\end{definition}

Clearly, for $\Omega = \Hn$ the quasihyperbolic metric is the same as the hyperbolic metric, but in the case of the ball it is not. The quasihyperbolic metric of the ball is connected to the hyperbolic metric by the following inequalities (see \cite[3.3]{vu2}):
\[
\rho_{\Bn}(x,y) \le 2 k_{\Bn}(x,y) \le 2\rho_{\Bn}(x,y),\qquad x,y\in \Bn.
\]

Obviously, it follows that quasihyperbolic metric is not, in general, invariant under conformal mappings, not even in M\"obius transformations. However, we have the following result:

\begin{proposition}
\cite[3.10]{vu2}
If $\Omega,\Omega'$ are proper subdomains of $\Rn$ and $f\colon \Omega\to \Omega'=f(\Omega)$ is a M\"obius transformation, then
\[
\frac{1}{2}k_\Omega(x,y) \le k_{\Omega'}\big(f(x),f(y)\big)\le 2 k_\Omega(x,y),
\]
for all $x,y\in \Omega$.
\end{proposition}

The proof of the above was first given by Gehring and Palka \cite{gp}, where a generalization of the result for quasiconformal mappings was also obtained.

In the case of the hyperbolic metric, it is easy to characterize the geodesics of this metric at least in the case of the ball and half-space. The properties, or even existence, of quasihyperbolic geodesics is not immediately clear, and they have been studied by several authors. In $\Rn$ it is, however, true that quasihyperbolic metric is always geodesic. Even this is not true in the more general setting of Banach spaces which we will consider later in this paper.

Their results already reveal that the quasihyperbolic metric is useful in study of geometric geometric function theory, but one crucial problem remains. While the quasihyperbolic metric is easy to define, we do not yet have any effective estimates for it. For this reason, we need yet another definition.

\begin{definition}
The \emph{distance ratio metric} or \emph{$j$-metric} in a proper subdomain $\Omega$ of the Euclidean space $\Rn$, $n \ge 2$, is defined by
\[
  j_\Omega(x,y) = \log \left( 1+\frac{|x-y|}{\min \{ d(x),d(y) \}}\right),
\]
where $d(x)$ is the Euclidean distance between $x$ and $\partial \Omega$. 
\end{definition}

Again, if the domain $\Omega$ is clear from the context, we use notation $j$ instead of $j_\Omega$. The distance ratio metric by F.W. Gehring and B.G. Osgood in 1979 \cite{go}.  The above form for the distance ratio metric was introduced by M. Vuorinen \cite{vu1}.

A useful inequality connecting this the $j$-metric to the quasihyperbolic metric is the following \cite[Lemma 2.1]{gp}:
\[
k_\Omega(x,y) \ge j_\Omega(x,y),\qquad x,y\in \Omega.
\]
An inequality to the other direction does not, in general, hold. For example, one may consider the slit plane $\Omega = \C \setminus \R_+$, and two points $s\pm ti$, where $s,t>0$. When $s\to +\infty$ and $t$ remains fixed the $j$-metric distance of the points does not change, but the quasihyperbolic distance of the points goes to infinity. If the inequality
\[
k_\Omega(x,y) \le c\,j_\Omega(x,y),\qquad x,y\in \Omega,
\]
holds for all $x,y\in \Omega$ where $c\ge 1$ is a constant, then we say that the domain $\Omega$ is \emph{uniform}. There are several equivalent ways of defining uniformity of a domain. Uniformity of different domains and the respective constants of uniformity have been studied by numerous authors, see e.g. \cite{vu2, lin}.

For a domain $G \subsetneq \Rn$ and a metric $m \in \{ k_G,j_G \}$ we define the \emph{metric ball} (\emph{metric disk} in the case $n = 2$) for $r > 0$ and $x \in G$ by
\[
  \Ball{m}{x}{r} = \{ y \in G \colon m(x,y) < r \}.
\]
We call $\Ball{k}{x}{r}$ the \emph{quasihyperbolic ball} and $\Ball{j}{x}{r}$ the \emph{$j$-metric ball}. A domain $G \subsetneq \Rn$ is \emph{starlike with respect to} $x \in G$ if for all $y \in G$ the line segment $[x,y]$ is contained in $G$ and $G$ is \emph{strictly starlike with respect to} $x$ if each half-line from the point $x$ meets $\partial G$ at exactly one point. If $G$ is (strictly) starlike with respect to all $x \in G$ then it is \emph{(strictly) convex}. A domain $G \subsetneq \Rn$ is \emph{close-to-convex} if $\Rn \setminus G$ can be covered with non-intersecting half-lines. By half-lines we mean sets $\{ x \in \Rn \colon x = t y+z, \, t > 0 \}$ and $\{ x \in \Rn \colon x = t y+z, \, t \ge 0 \}$ for $z \in \Rn $ and $y \in \PS$. We define metric balls, convexity and starlikeness similarly in Banach spaces.

Clearly convex domains are starlike and starlike domains are close-to-convex as well as complements of close-to-convex domains are starlike with respect to infinity. However, close-to-convex sets need not be connected. An example of a close-to-convex disconnected set is the union of two disjoint convex domains. We use notation $B^n(x,r)$ and $S^{n-1}(x,r)$ for Euclidean balls and spheres, respectively, with radius $r > 0$ and center $x \in \Rn$.

\section{Metric balls in subdomains of $\Rn$}

In this section we consider properties of the quasihyperbolic metric and $j$-metric balls in subdomains of $\Rn$.

For the quasihyperbolic metric the explicit formula is known only in a few special domains like half-space, where it agrees with the usual hyperbolic metric, and punctured space, where the formula was introduced by G.J. Martin and B.G. Osgood in 1986 \cite{mo}. They proved that for $x,y \in \PS$ and $n \ge 2$
\begin{equation}\label{MOformula}
  k_\PS(x,y) = \sqrt{\alpha^2+\log^2\frac{|x|}{|y|}},
\end{equation}
where $\alpha \in [0,\pi]$ is the angle between line segments $[x,0]$ and $[0,y]$ at the origin. This implies that the quasihyperbolic geodesics in $\PS$ are logarithmic spirals, circular arcs or line segments. All other domains, where the explicit formula for the quasihyperbolic is known are derived from the half-space and the punctured space.

For any domain $\Omega \subset \Rn$ and points $x,y \in \Omega$ there exists a quasihyperbolic geodesic \cite{go}. This implies that the quasihyperbolic balls in $\Omega$ are always connected. In convex plane domains $\Omega \subset \R^2$ the quasihyperbolic circles $\partial B_k$ are always $C^1$ smooth Jordan curves \cite[Corollary 5.14]{v1}. In a general domain $\Omega \subset \Rn$ the quasihyperbolic ball need not be $C^1$ smooth. However, the quasihyperbolic balls are smooth except from possible inwards pointing cusps \cite[Theorem 5.10]{v1}.

V\"ais\"al\"a posed \cite[p. 448]{v2} the following conjectures:
\begin{conjecture}
  Let $\Omega \subset \Rn$ and $x,y \in \Omega$.
  \begin{itemize}
    \item[(1)] Uniqueness conjecture: There is a universal constant $c_u > 0$ such that if $k(x,y) < c_u$, then there is only one quasihyperbolic geodesic from $x$ to $y$.\\
    \item[(2)] Prolongation conjecture: There is a universal constant $c_p > 0$ such that if $\gamma$ is a quasihyperbolic geodesic from $x$ to $y$ with $\ell_k(\gamma) =k(x,y) < c_p$, then there is a quasihyperbolic geodesic $\gamma'$ from $x$ to $y'$ such that $\gamma \subset \gamma'$ and $\ell (\gamma') = c_p$.\\
    \item[(3)] Convexity conjecture: There is a universal constant $c_c > 0$ such that the quasihyperbolic ball $B_k(x,r)$ is strictly convex for all $x \in \Omega$ and $r < c_c$.
  \end{itemize}
\end{conjecture}
V\"ais\"al\"a also proved that for $\Omega \subset \Rn$ the Convexity conjecture implies the Uniqueness conjecture (with $c_u = 2c_c$), the Uniqueness conjecture implies the Prolongation conjecture (with $c_p = c_u \wedge \pi/2$) and for any $\Omega \subset \R^2$ the above conjectures hold with $c_u=2$, $c_p=2$ and $c_c=1$. Presumably the constant $c_u = 2$ is not sharp, see Example \ref{kinPP}, \cite[Remark 6.18]{k3}. In convex domains $\Omega \subset \Rn$ the quasihyperbolic geodesics are always unique \cite[Theorem 2.11]{MartioVaisala11} and moreover, the quasihyperbolic balls are always strictly convex \cite[2.13]{MartioVaisala11}.

The question of convexity of the hyperbolic type metric balls was posed by M. Vuorinen in 2006 \cite[8.1]{vu2}. Soon after that a series of papers \cite{k1,k2,k4,MartioVaisala11,RasilaTalponen,v2} was published on the problem in the case of the quasihyperbolic metric and the $j$-metric. We review the main results in $\Rn$.

\subsection{Quasihyperbolic metric}

We introduce some convexity properties of quasihyperbolic metric balls in the domain $\Omega = \PS$. For each result the idea is the following: first prove the result in the case $n=2$ by using \eqref{MOformula} and then generalize the result to $n>2$ by symmetry of the domain. For more details about the proofs see the original papers.

Before the convexity results we define two constants. The constant $\kappa$ is the solution of the equation
\begin{equation}\label{kappa}
    \cos \sqrt{p^2-1}+\sqrt{p^2-1}\sin \sqrt{p^2-1} = e^{-1}
\end{equation}
for $p \in [1,\pi]$, and the constant $\lambda$ is the solution of the equation
\begin{equation}\label{lambda}
  \cos \sqrt{p^2-1} + \sqrt{p^2-1} \sin \sqrt{p^2-1} = 0
\end{equation}
for $p \in (2,\pi)$.

\begin{theorem}
\label{quasiconv-k2}
  1) For $x \in \PS$ the quasihyperbolic ball $B_k(x,r)$ is strictly convex for $r \in (0,1]$ and it is not convex for $r > 1$.

  \noindent 2) For $x \in \PS$ the quasihyperbolic ball $B_k(x,r)$ is strictly starlike with respect to $x$ for $r \in (0,\kappa]$ and it is not starlike with respect to $x$ for $r > \kappa$, where $\kappa$ is defined by (\ref{kappa}) and has a numerical approximation $\kappa \approx 2.83297$.
\end{theorem}
\begin{proof}
  \cite[Theorem 1.1]{k2}.
\end{proof}

\begin{figure}[ht!]
  \begin{center}
    \includegraphics[width=5cm]{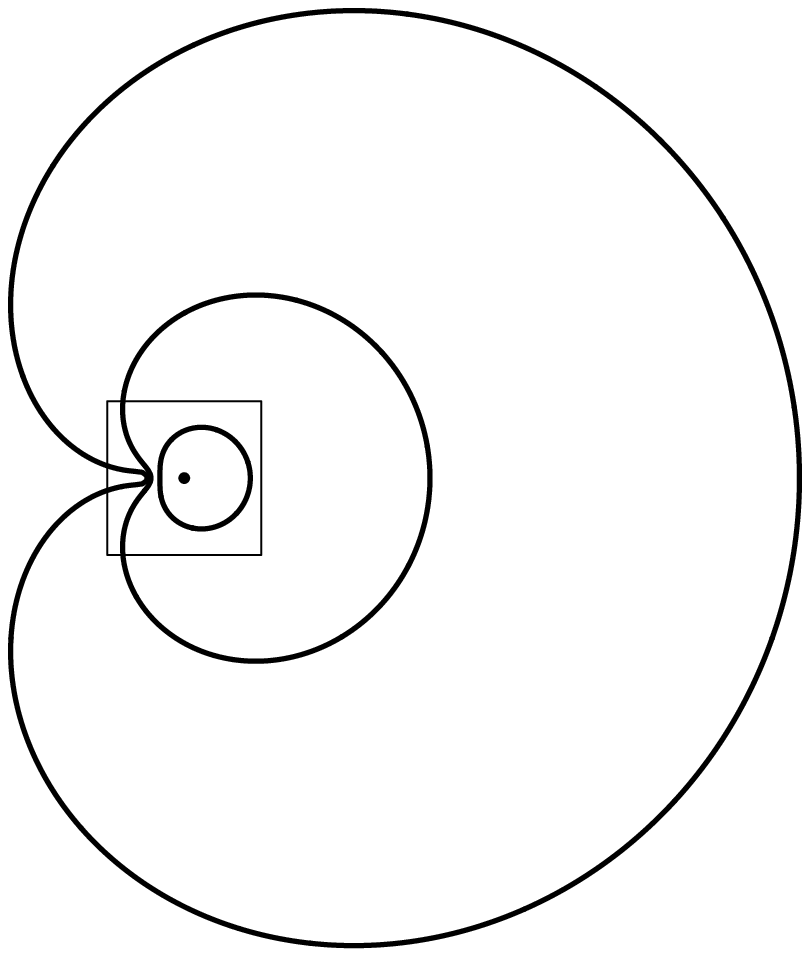}\hspace{1cm}
    \includegraphics[width=5cm]{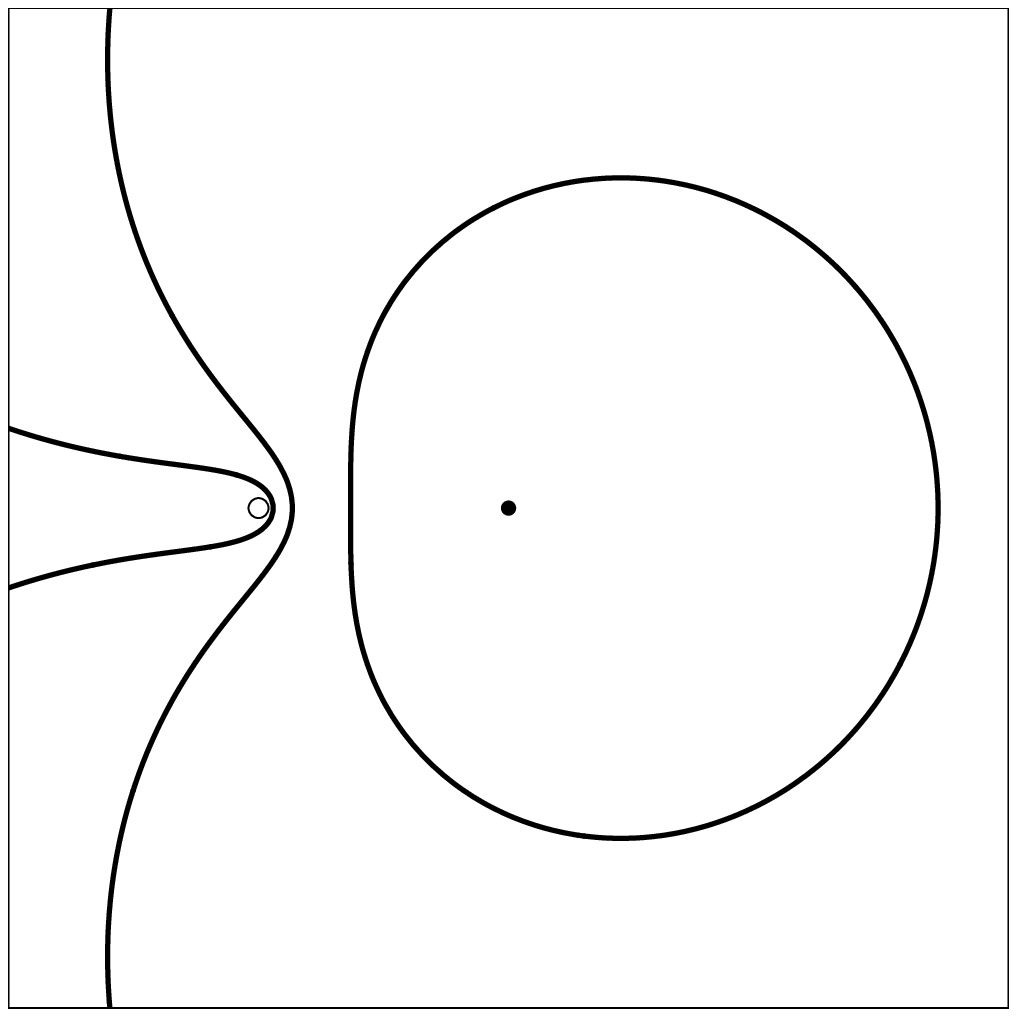}
    \caption{Boundaries of quasihyperbolic disks $B_k(x,r)$ in $\PP$ with radii $r=1$, $r=2$ and $r=\kappa$.}
  \end{center}
\end{figure}

\begin{theorem}
  If domain $\Omega \subsetneq \Rn$ is starlike with respect to $x \in \Omega$, then the quasihyperbolic ball $B_k(x,r)$ is starlike with respect to $x$ for all $r > 0$.
\end{theorem}
\begin{proof}
  \cite[Theorem 2.10]{k2}.
\end{proof}

\begin{theorem}
  For $y \in \PS$ the quasihyperbolic ball $\Ball{k}{y}{r}$ is close-to-convex, if $r \in (0,\lambda]$, where $\lambda$ is defined by (\ref{lambda}) and has a numerical approximation $\lambda \approx 2.97169$. Moreover, the constant $\lambda$ is sharp in the case $n = 2$.
\end{theorem}
\begin{proof}
  \cite[Theorem 1.2]{k3}.
\end{proof}

\subsection{Distance ratio metric}

We introduce some convexity properties of $j$-metric balls. The results are similar to the ones presented for the quasihyperbolic metric. However, now the results are valid for all subdomains of $\Rn$. The main idea in the proof of the results is to first prove the result in $\PS$ and then generalize it to general subdomains of $\Rn$. Generalization is fairly simple, because of the explicit formula for the metric is known. For more details about the proofs see the original papers.

\begin{theorem}\label{jradius1}
  For a domain $\Omega \subsetneq \Rn$ and $x \in \Omega$ the $j$-balls $B_j(x,r)$ are convex if $r \in (0,\log 2]$ and strictly starlike with respect to $x$ if $r \in \big( 0,\log (1+\sqrt 2) \big]$.
\end{theorem}
\begin{proof}
  \cite[Theorem 1.1]{k1}.
\end{proof}

\begin{theorem}\label{jradius2}
  Let $r > 0$, $\Omega \subsetneq \Rn$ be a convex domain and $x \in \Omega$. Then $j$-balls $B_j(x,r)$ are convex.
\end{theorem}
\begin{proof}
  \cite[Theorem 4.5]{k1}.
\end{proof}

Results of Theorems \ref{jradius1} and \ref{jradius2} are illustrated in Figure \ref{jballs}.

\begin{figure}[htp]
  \begin{center}
    \includegraphics[width=5cm]{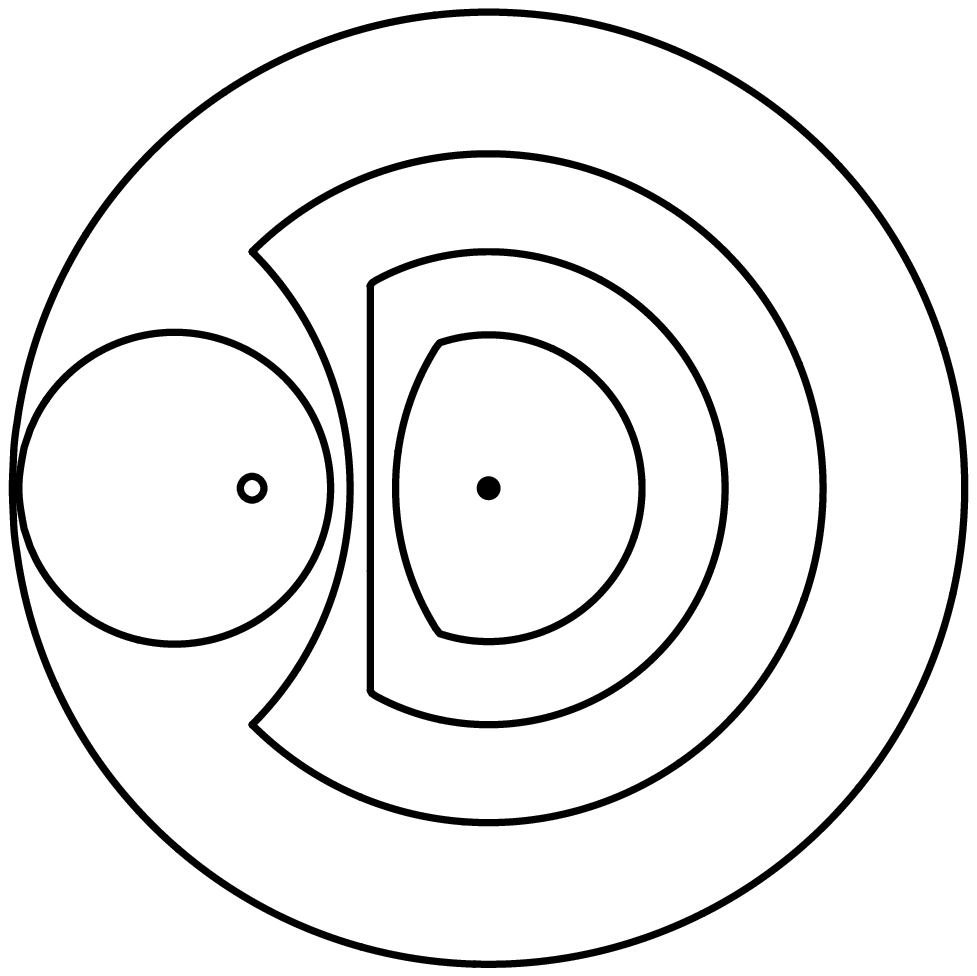}\hspace{1cm}
    \includegraphics[width=5cm]{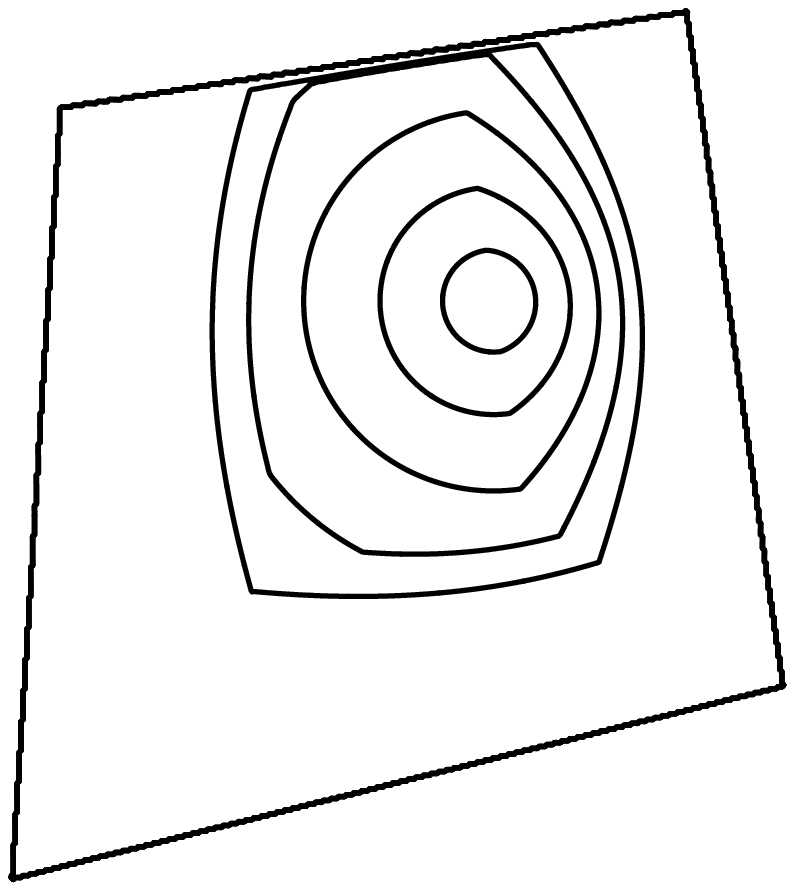}
    \caption{Left: boundaries of $j$-disks $j(1,r)$ in punctured plane $\Omega = \PP$ with $r=01/2$, $r=\log 2$, $r=\log(1+\sqrt{2})$ and $r=1.1 \approx \log 3$. Right: boundaries of $j$-disks in a convex polygonal domain.\label{jballs}}
  \end{center}
\end{figure}

\begin{theorem}
  Let $r > 0$ and $\Omega \subsetneq \Rn$ be a starlike domain with respect to $x \in \Omega$. Then the $j$-balls $B_j(x,r)$ are starlike with respect to $x$.
\end{theorem}
\begin{proof}
  \cite[Theorem 4.8]{k1}.
\end{proof}

\begin{theorem}\label{connectedjballs}
  For a domain $\Omega \subsetneq \Rn$ and $x \in \Omega$ the $j$-metric ball $\Ball{j}{x}{r}$ is close-to-convex and connected, if $r \in (0, \log(1+\sqrt{3})]$. Moreover, the constant $\log(1+\sqrt{3})$ is sharp in the case $n = 2$.
\end{theorem}
\begin{proof}
  \cite[Theorem 1.2]{k3}.
\end{proof}

\subsection{Examples in subdomains of $\Rn$}

We consider some concrete examples of the quasihyperbolic metric and $j$-metric balls in $\Rn$.

\begin{example}\label{kinPP}
  Let us consider the quasihyperbolic distance in the punctured plane $\Omega=\PP$. By \cite{v2} quasihyperbolic geodesics between $x,y \in \Omega$ are unique whenever $k(x,y) \le 2$. Choosing $x \in \Omega$ and $y=-x$ it is by \eqref{MOformula} easy to see that $k(x,y) = \pi$. Since $|x|=|y|$ quasihyperbolic geodesic joining $x$ and $y$ is a subset of $S^1(0,|x|)$ and by selection of $x$ and $y$ it is clear that there are (at least) two different geodesics from $x$ to $y$.
\end{example}

While the metric space $(\Omega,k_\Omega)$, $\Omega \subset \Rn$, is always geodesic, the metric space $(\Omega,j_\Omega)$ is never geodesic \cite[Theorem 2.10]{k1}. This implies that quasihyperbolic balls are always connected and $j$-metric balls need not be connected.

\begin{example}
  An example of disconnected $j$-metric disks is shown in Figure \ref{disconnectedjdisks}. The domain is constructed from two disconnected Euclidean disks by joining then with a narrow corridor. In this domain it is possible to have a $j$-metric disk $B_j(x,r)$ such that $B_j(x,r)$ is connected and $\partial B_j(x,r)$ is disconnected. For more details see \cite[Remark 4.9]{k1}.
\end{example}

\begin{example}
  For any $N > 0$ it is possible to construct a domain such that there exists a $j$-metric ball with exactly $N$ components \cite[6.3]{k3}. It can be shown that the quasihyperbolic diameter of each of these components is bounded above by a constant depending only on dimension $n$ and radius $r$ \cite[Lemma 6.8]{k3}.
\end{example}

\begin{example}
  By Theorem \ref{connectedjballs} $j$-metric balls are always connected, if the radius is less than or equal to $\log (1+\sqrt{3})$. It is possible to find examples of $j$-metric balls such that they are disconnected for radius $\log (1+\sqrt{3}) + \varepsilon$ for all $\varepsilon > 0$. For example $B_j(\sqrt{3} e_2,r)$ in $\Omega=\R^2 \setminus \{ e_1,-e_1 \}$ is disconnected for all $r = \log (1+\sqrt{3}) + \varepsilon$, where $\varepsilon > 0$ is small enough \cite[Remark 2.10]{k4}.
\end{example}

\begin{figure}[htp]
  \begin{center}
    \includegraphics[width=7cm]{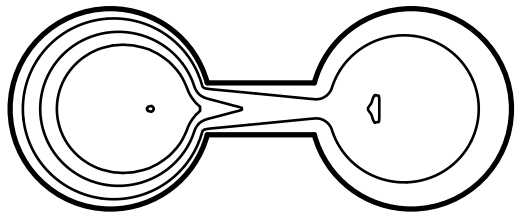}
     \caption{An example of disconnected $j$-disks.\label{disconnectedjdisks}}
  \end{center}
\end{figure}

\section{Quick tour on Banach spaces and their geometry}
\subsection{Basic notions}

Recall that a \emph{norm} on a (real or complex) vector space $V$ is a function $\rho\colon V\to [0,\infty)$ such that 
\[\rho(x+y)\leq \rho(x)+\rho(y),\quad x,y\in V\]
\[\rho(ax)=|a|\ \rho(x),\quad a\in \R,\ x\in V\]
\[\rho(x)=0\implies x=0,\quad x\in V.\]

Recall that a normed space $X=(X,\|\cdot\|)$ (where $X$ is a vector space and $\|\cdot\|$ is a norm) is said to be a Banach space if all of its Cauchy sequences converge. This means that for each Cauchy sequence, i.e. 
a sequence $(x_{n})\subset X$ such that $\limsup_{n,m\to \infty}\|x_{n}-x_{m}\|=0$ there is a (unique) point $x\in \X$ such that $\lim_{n\to\infty}\|x-x_{n}\|=0$.
Each finite-dimensional normed space, e.g. $(\R^{n}, | \cdot |)$, is a Banach space. Next we will give examples of the most simple Banach spaces. We denote by $\ell^{0}$ the vector space of all 
sequences of real numbers $(x_{1},x_{2},x_{3},\ldots)$ (this is not a normed space). Now, we will define linear subspaces of $\ell^{0}$ by taking all the sequences $(x_{n})$ 
such that $\|(x_{n})\|<\infty$, where the norm $\|\cdot\|$ varies according to the choice of the space. Thus we obtain the following classical spaces:
\[\ell^{\infty}=(\ell^{\infty},\|\cdot\|_{\infty}),\ \mathrm{where}\ \|(x_{n})\|_{\infty}=\sup_{n}|x_{n}|,\]
\[\ell^{p}=(\ell^{p},\|\cdot\|_{p}),\ 1\leq p<\infty,\ \mathrm{where}\ \|(x_{n})\|_{p}=\left(\sum_{n=1}^{\infty}|x_{n}|^{p}\right)^{\frac{1}{p}}.\] 

An example of a normed space, which is not a Banach space is the normed subspace of $\ell^{\infty}$ consisting of all finitely supported vectors, called $c_{00}$. This is not 
a Banach space because vectors of the form $(1,1/2,1/3,\ldots, 1/n,0,0,0,\ldots)$ ($n$ runs in $\N$) form a Cauchy sequence but the point of convergence\\ 
$(1,1/2,\ldots,1/n,1/n+1,\ldots)\in \ell^{\infty}$ is not in $c_{00}$. However, each normed space can be completed (essentially in a unique way) to be a Banach space. For example, the unique completion of $c_{00}$ is the Banach space
\[c_{0}=(c_{0},\|\cdot\|_{\infty}),\ \mathrm{where}\ c_{0}=\{(x_{n})\in \ell^{\infty}:\ \lim_{n\to\infty}|x_{n}|=0\}.\]

In a similar spirit, the Banach spaces of Lebesgue measurable functions\\ $f\colon [0,1]\to \R$ are defined: $L^{\infty}=(L^{\infty},\|\cdot\|_{L^{\infty}})$, where $\|f\|_{L^{\infty}}$ is the essential supremum of
\[|f|=\inf\{a>0:\ m(\{t\in [0,1]:\ |f(t)|>a\})=0\},\]
\[L^{p}=(L^{p},\|\cdot\|_{L^{p}}),\ 1\leq p<\infty,\ \mathrm{where}\ \|f\|_{L^{p}}=\left(\int_{0}^{1}|f(t)|^{p}\ dt\right)^{\frac{1}{p}}.\]

The spaces $\ell^{2}$ and $L^{2}$ in fact isometrically isomorphic. These \emph{Hilbert spaces} are infinite-dimensional generalizations of the Euclidean space and they have a central role in functional analysis and applications 
in many branches of mathematics. 

The Banach space of continuous functions $f\colon K\to \R$ with point-wise linear operations and the sup-norm $\|f\|_{\mathrm{sup}}=\sup_{t\in K}|f(t)|$ is denoted by $C(K)$, where $K$ is a compact Hausdorff space.

The Hardy spaces $H^{p}(\D)$ consist of holomorphic functions from the open unit disk $\D$ to the complex plane $\C$, where the norm is given by
\[\|f\|_{H^{p}}=\left(\frac{1}{2\pi}\int_{0}^{2\pi}\lim_{r\to 1}|f(re^{i\theta})|^{p}\ d\theta \right)^{\frac{1}{p}}.\]

\subsection{Duality, reflexive spaces}
The dual space of a real Banach space $X=(X,\|\cdot\|_{X})$ consists of all continuous linear mappings $F\colon X\to \R$. These mappings are called \emph{functionals}.
The dual of a Banach space $X$, denoted by $X^{\ast}$, is equipped with the norm $\|x^{\ast}\|_{\X^{\ast}}=\sup_{y\in X,\ \|y\|_{\X}=1}x^{\ast}[y]$ and it is also a Banach space. 
For example, it can be seen by using H\"older's inequality that for $1<p<\infty$ the dual space of $L^{p}$ is $L^{q}$ and the dual of $\ell^{p}$ is $\ell^{q}$, where $1<q<\infty$ satisfies $1/p + 1/q=1$.
The duality is in the sense that $f^{\ast}[g]=\int_{0}^{1}f^{\ast}(t)g(t)\ dt$ or $x^{\ast}[y]=\sum_{n=1}^{\infty}x^{\ast}(n)y(n)$, respectively.

The coarsest topology on $X$, that makes all the dual functionals continuous, is called the \emph{weak topology}. 
Recall that a topological space is \emph{compact} if each its open cover admits a finite subcover. For example, the closed balls of $\R^n$ are compact.
The closed unit ball 
\[B_{X}=\{x\in X:\ \|x\|\leq 1\}\]
is never compact in the topology given by the norm, if $X$ is infinite-dimensional. However, it might happen that $B_{X}$ is compact in the weak topology. In such a case $X$ is said to be \emph{reflexive}.
This happens exactly when $X$ is isometric in a canonical way to its second dual $X^{\ast\ast}=(X^{\ast})^{\ast}$. For example, the dual of $c_{0}$ is $\ell^{1}$, whose dual in turn is $\ell^{\infty}$. For this reason 
$c_{0}$ is not reflexive, since $c_{0}$ is not isometric to $\ell^{\infty}$ (in any way). Compactness of the unit ball in weak topology is very useful property and it has many applications in several branches of mathematics, e.g. 
for proving the existence of quasihyperbolic geodesics on convex subsets of a reflexive Banach space (see e.g. \cite{Vaisala05}). We will give subsequently some geometric conditions ensuring that a space is reflexive.

\subsection{The Radon-Nikodym principle, differentiation and integration in Banach spaces}

In calculations regarding paths and the quasihyperbolic metric it is often very convenient to write a path as an integral of its derivative in some sense. However, in a general Banach space this is not always 
possible, and even if it is, some extra care is needed in doing calculus. We require a notion of derivative for mappings on Banach spaces as well as a Banach-valued integral. The fundamental theorem of calculus in the form of 
the Radon-Nikodym theorem does not hold in all Banach spaces (e.g. not in $c_{0}$) and in principle one is required to study Banach-valued measures in order to decide which space satisfy the Radon-Nikodym theorem.
The good news is that there is a wide class of Banach spaces, namely ones with the so-called \emph{Radon-Nikodym Property} (RNP), in which one can apply the fundamental theorem of calculus principle.
 
Let us begin by defining the two most important derivatives of a function on a Banach space.  
Suppose that $f\colon \X\to \Y$ is a continuous mapping between Banach spaces $\X$ and $\Y$. The functional analytic differentiation is a a straight forward generalization of the usual one in $\R^n$.
Namely, we say that $f$ is \emph{G\^ateaux differentiable} at $x_{0}\in \X$ if there is a linear continuous mapping $f^{\prime}(x_{0})\colon \X\to \Y$ given by
\[f^{\prime}(x_{0})(z)=\lim_{t\to 0}\frac{f(x_{0}+tz)-f(x_{0})}{t}\]
for $z\in \X$. If the G\^ateaux derivative $f^{\prime}(x_{0})\colon \X\to \Y$ exists at point $x_{0}$ and 
\[\lim_{r\to 0^{+}}\sup_{h\in \X,\ \|h\|=r}\frac{f^{\prime}(x_{0})(h)-(f(x_{0}+h)-f(x_{0}))}{\|h\|}=0,\]
then $f$ is said to be \emph{Fr\'echet differentiable} at $x_{0}$. In $\R^n$ these concepts of differentiability coincide due to compactness of closed balls.
 
A Banach-valued function $f\colon [0,1]\to \X$ is said to be \emph{measurable} if there is a sequence $(f_{n})$ of countably valued functions of the 
form 
\[f_{n}\colon [0,1]\to \X,\ f_{n}(t)=\sum_{i=1}^{\infty}x_{i,n}\chi_{A_{i,n}}(t),\quad \mathrm{with}\ A_{i,n}\subset [0,1]\ \mathrm{measurable},\] 
such that $f_{n}(t)\to f$ as $n\to \infty$ for almost every $t\in [0,1]$.

There is a natural way to define an integral $\int_{0}^{1}f(t)\ dt$ of a measurable function $f$, called the \emph{Bochner integral}:
\[\int_{0}^{1}f(t)\ dt=\lim_{n\to\infty}\sum_{i=1}^{\infty} m(A_{i,n})x_{i,n},\]
where the limit is taken in the norm, if 
\[\lim_{n\to\infty}\int_{0}^{1}\|f(t)-f_{n}(t)\|\ dt=0.\]
This definition is well set as it does not depend on the selection of the sequence $(f_{n})$.

A Banach space $\X$ has the RNP if and only if each Lipschitz path $\gamma\colon [0,1]\to \X$ is G\^ateaux differentiable almost everywhere. In such a case one can represent 
such a path $\gamma$ as a Bochner integral of its G\^ateaux derivative by the formula 
\[\gamma(s)=\gamma(0)+\int_{0}^{s}\gamma^{\prime}(t)\ dt,\]
where $\int_{0}^{s}f(t)\ dt=\int_{0}^{1}\chi_{[0,s]}(t)f(t)\ dt$.
 
Reflexive spaces and separable dual spaces have the RNP. The RNP is a kind of convexity property, which does not allow too much flatness on the unit sphere. 
Given two Banach space properties $P$ and $Q$, we say that $P$ is a \emph{dual property} of $Q$ if $P$ of $\X^{\ast}$ implies $Q$ of $\X$. For example, RNP is a dual property for another property of Banach spaces,
called Asplund. A Banach space is Asplund if each of its separable subspaces has separable dual. As noted, the RNP is a kind of convexity property, whereas Asplund is a kind of smoothness property, which yields that 
the norm is a Fr\'echet differentiable function in a dense set of points.  
 
\subsection{Uniform convexity and smoothness of the unit ball}
The unit ball of $(\R^{2},\|\cdot\|_{\infty})$ is a non-typical one, it is a square, and thus it barely deserves to be called a ball. Two things to observe are that it has flat sides and sharp corners.
This is to say that the space is not \emph{strictly convex} and not \emph{smooth}. In order to bypass some pathological behaviour of Banach spaces, one would like to ensure convexity and smoothness of the unit ball, 
not just looking at some particular points (like corners above) but instead the whole asymptotics of the ball. This is in particular necessary when one desires to form estimates involving smoothness of the norm 
non-specific to a certain location of the unit ball.

Next, we will present two moduli related to the geometry of Banach spaces. The \emph{modulus of convexity} $\delta_{\X}(\epsilon),\ 0<\epsilon \leq 2,$ is defined by
\[\delta_{\X}(\epsilon):=\inf\{1-\|x+y\|/2:\ x,y\in \X,\ \|x\|=\|y\|=1,\ \|x-y\|=\epsilon\},\]
and the \emph{modulus of smoothness} $\rho_{\X}(\tau),\ t>0$ is defined by 
\[\rho_{\X}(\tau):=\sup\{(\|x+y\|+\|x-y\|)/2 -1,\ x,y\in \X,\ \|x\|=1,\ \|y\|=\tau\}.\]
The Banach space $\X$ is called \emph{uniformly convex} if $\delta_{\X}(\epsilon)>0$ for all $\epsilon>0$, and \emph{uniformly smooth}  if 
\[
\lim_{\tau\to 0^{+}}\frac{\rho_{\X}(\tau)}{\tau}=0.
\]
The modulus of convexity measures the maximum distance to the boundary form a point of a cord of length $\epsilon$ inside the unit ball. Thus the unit ball of a uniformly convex space is round (meaning non-flat) 
in all directions and this roundness is controlled by the modulus $\delta$. In getting acquainted with the notion of uniform smoothness, it is perhaps worthwhile to calculate the modulus of smoothness for $\ell^{1}$, 
which is the function $\tau \mapsto \tau$.

Moreover, a space $\X$ is uniformly convex (resp. uniformly smooth) of power type $p\in [1,\infty)$ if $\delta_{\X}(\epsilon)\geq K\epsilon^{p}$ (resp. $\rho_{\X}(\tau)\leq K\tau^p$) for some $K>0$.
The power type $p$ relates the smoothness/convexity of the space $\X$ to that of the spaces $L^{p}$. The power types $q$ of smoothness and $p$ of convexity satisfy 
$q\leq 2\leq p$. In fact, Hilbert spaces (corresponding to the case $q=p=2$) are both uniformly convex and uniformly smooth and moreover have \emph{the} best possible modulus functions of convexity and smoothness. 
Spaces linearly homeomorphic to a uniformly smooth or a uniformly convex space are (super)reflexive.

An example of a Banach space, which is reflexive, strictly convex and smooth, but not superreflexive, is 
\[\ell^{2}\oplus_{2}\ell^{3}\oplus_{2}\ell^{4}\oplus_{2}\ldots.
\]

\section{Metric balls on Banach spaces}

Some previous work will be treated here (\cite{Martin85}, \cite{MartioVaisala11}, \cite{Vaisala90}, \cite{Vaisala91}, \cite{Vaisala92}, \cite{Vaisala98}, \cite{Vaisala99}, \cite{Vaisala05}).

Next we will summarize the main results in \cite{RasilaTalponen}. First, in the $j$-metric, balls are starlike for radii $r\le \log 2$:

\begin{theorem}\label{thm: sl}
Let $\X$ be a Banach space, $\Omega\subsetneq\X$ a domain, and let $j$ be the distance ratio metric on $\Omega$. Then each $j$-ball  $\B_{j}(x_{0},r),\ x_{0}\in \Omega,$ is starlike for radii $r\leq \log 2$.
\end{theorem}

The following theorem is a generalization of a result of Martio and V\"ais\"al\"a \cite[2.13]{MartioVaisala11}:

\begin{theorem}\label{thm: cd} 
Let $\X$ be a Banach space and $\Omega\subsetneq\X$ a convex domain. Then all quasihyperbolic balls and $j$-balls on $\Omega$ are convex. Moreover, if $\Omega$ is uniformly convex, or if $\X$ is strictly convex and has the RNP, then
these balls are strictly convex.
\end{theorem}

If $\Omega$ is starlike with respect to $x_0\in \Omega$, then all quasihyperbolic and $j$-metric balls centered at $x_0$ are starlike as well.

\begin{theorem}\label{thm_starlike_starlike}
Let $\X$ be Banach space, $x_{0}\in \X$ and let $\Omega\subset \X$ be a domain which is starlike with respect to $x_{0}$. Then all balls $\B_{j}(x_{0},r)$ and $\B_{k}(x_{0},r)$ of $\Omega$ are starlike.
\end{theorem}

Finally, we note that in \cite{k2} (see Theorem \ref{quasiconv-k2}) critical radii are provided for the convexity of quasihyperbolic and $j$-balls on punctured $\R^n$. Again, it is a natural question whether the existence of such radii can be established, in the Banach space setting

In order to check the convexity of a $k$-ball in a punctured space it would seem natural to exploit an averaging argument similar to the proof of Theorem \ref{thm: cd}. We have obtained the following partial result which comes very close to providing such a device:

\begin{theorem}\label{lm: QHlemma}
Let $\X$ be a Banach space, which is uniformly convex and uniformly smooth, both moduli being of power type $2$. We consider the quasihyperbolic metric $k$ on $\Omega=\X\setminus \{0\}$. Then there exists $R>0$ as follows.
Assume that $\gamma_{1},\gamma_{2}\colon [0,t_{2}]\to \Omega$ are rectifiable paths satisfying the following conditions:
\begin{enumerate}
\item[(i)]{$\gamma_{1},\gamma_{2}$ and $(\gamma_{1}+\gamma_{2})/2$ are contained in $\B_{\|\cdot \|}(0,2)\setminus \B_{\|\cdot \|}(0,1)$},
\item[(ii)]{$\gamma_{1}(0)=\gamma_{2}(0)$},
\item[(iii)]{$\ell_{k}(\gamma_{1})\vee \ell_{k}(\gamma_{2})\leq R$},
\item[(iv)]{$\ell_{\|\cdot\|}(\gamma_{1})=t_{1}\leq t_{2}=\ell_{\|\cdot\|}(\gamma_{2})$},
\item[(v)]{the paths are parameterized with respect to $\ell_{\|\cdot\|}$, except that\\ $\gamma_{1}(t)=\gamma_{1}(t_{1})$ for $t\in [t_{1},t_{2}]$.}
\end{enumerate}
Then the following estimate holds: 
\[
\frac{\ell_{k}(\gamma_{1})+\ell_{k}(\gamma_{2})}{2}+\int_{t_{1}}^{t_{2}}\frac{\|d\gamma_{2}\|}{2d(\gamma_{2})}
\geq \ell_{k}\left(\frac{\gamma_{1}+\gamma_{2}}{2}\right)+\int_{0}^{t_{1}}\frac{\delta_{\X}(\|D(\gamma_{1}-\gamma_{2})\|)}{\|\gamma_{1}\|+\|\gamma_{2}\|}\ ds.
\] 
\end{theorem}

\section{Summary of critical radii}

In Table \ref{ksummary} and Table \ref{jsummary} the known radii of convexity, starlikeness and close-to-convexity for the quasihyperbolic and the $j$-metric balls are presented.

\begin{table}[ht]
  \begin{center}\begin{tabular}{c|c|c|c}
    domain & convex & starlike w.r.t $x$ & close-to-convex\\
    \hline
    $\PP$ & 1 \cite{k1} & $\kappa \approx 2.83$ \cite{k1} & $\lambda \approx 2.97$ \cite{k4}\\
    $\PS$ & 1 \cite{k1} & $\kappa \approx 2.83$ \cite{k1} & $\lambda^* \approx 2.97$ \cite{k4}\\
    convex, $\Rn$ & $\infty$ \cite{MartioVaisala11} & $\infty$ \cite{MartioVaisala11} & $\infty$ \cite{MartioVaisala11}\\
    convex, BS & $\infty$ \cite{RasilaTalponen} & $\infty$ \cite{RasilaTalponen} & $\infty$ \cite{RasilaTalponen}\\
    starlike w.r.t. $x$, $\Rn$ & ? & $\infty$ \cite{k1} & $\infty$ \cite{k1}\\
    general ($n = 2$), $\Rn$ & 1 \cite{v2} & $\pi/2^*$ \cite{v1} & $\pi/2^*$ \cite{v1}\\
    general ($n \ge 2$), $\Rn$ & ? & $\pi/2^*$ \cite{v1} & $\pi/2^*$ \cite{v1}\\
  \end{tabular}\end{center}
  \caption{\label{ksummary} The known radii of convexity, starlikeness and close-to-convexity for the quasihyperbolic balls $\Ball{k}{x}{r}$. Notation $r^*$ means that the radius $r$ is not sharp.}
\end{table}
\begin{table}[htr]
  \begin{center}\begin{tabular}{c|c|c|c}
    domain & convex & starlike w.r.t $x$ & close-to-convex\\
    \hline
    convex, $\Rn$ & $\infty$ \cite{k2} & $\infty$ \cite{k2} & $\infty$ \cite{k2}\\
    convex, BS & $\infty$ \cite{RasilaTalponen} & $\infty$ \cite{RasilaTalponen} & $\infty$ \cite{RasilaTalponen}\\
    starlike w.r.t. $x$, $\Rn$ & $\log 2$ \cite{k2} & $\infty$ \cite{k2} & $\infty$ \cite{k2}\\
    general ($n = 2$), $\Rn$ & $\log 2$ \cite{k2} & $\log (1+\sqrt{2})$ \cite{k2} & $\log (1+\sqrt{3})$ \cite{k4}\\
    general ($n \ge 2$), $\Rn$ & $\log 2$ \cite{k2} & $\log (1+\sqrt{2})$ \cite{k2} & $\log (1+\sqrt{3})^*$ \cite{k4}\\
  \end{tabular}\end{center}
  \caption{\label{jsummary} The known radii of convexity, starlikeness and close-to-convexity for the $j$-metric balls $\Ball{j}{x}{r}$. Notation $r^*$ means that the radius $r$ is not sharp.}
\end{table}

\comment{ 

\section{Open problems}

\begin{enumerate}
  \item Is it true that in the Uniqueness conjecture $c_u = \pi$ for $G \in \R^2$?
  \item ...
\end{enumerate}

} 


\end{document}